\documentclass{amsart}
\usepackage{amssymb,latexsym,amsmath,amsthm, graphicx, epsfig}
\usepackage{tikz}
\usepackage[enableskew]{youngtab}

\usepackage[normalem]{ulem} 
\usepackage{soul}

\newtheorem{theorem}{Theorem}[section]
\newtheorem{corollary}[theorem]{Corollary}
\newtheorem{lemma}[theorem]{Lemma}

\newtheorem{example}[theorem]{Example}
\newtheorem{conjecture}[theorem]{Conjecture}

\makeatletter
\newcommand{\pushright}[1]{\ifmeasuring@#1\else\omit\hfill$\displaystyle#1$\fi\ignorespaces}
\newcommand{\pushleft}[1]{\ifmeasuring@#1\else\omit$\displaystyle#1$\hfill\fi\ignorespaces}
\makeatother

\allowdisplaybreaks

\begin{document}

\title{On simultaneous $(s, s+t, s+2t, \dots)$-core partitions}

\author{William J. Keith}
\address{Department of Mathematical Sciences, 
Michigan Technological University, Houghton, MI 49931, USA}
\email{wjkeith@mtu.edu}

\author{Rishi Nath}
\address{Department of Mathematics and Computer Science, 
York College, CUNY, Jamaica, NY 11451, USA}
\email{rnath@york.cuny.edu}

\author{James A. Sellers}
\address{Department of Mathematics and Statistics, 
University of Minnesota Duluth, Duluth, MN 55812, USA}
\email{jsellers@d.umn.edu}



\date{}

\maketitle

\begin{abstract}
We consider simultaneous $(s,s+t,s+2t,\dots,s+pt)$-core partitions in the large-$p$ limit, or (when $s<t$), partitions in which no hook may be of length $s \pmod{t}$.  We study generating functions, containment properties, and congruences when $s$ is not coprime to $t$.  As a boundary case of the general study made by Cho, Huh and Sohn, we provide enumerations when $s$ is coprime to $t$, and answer positively a conjecture of Fayers on the polynomial behavior of the size of the set of simultaneous $(s,s+t,s+2t,\dots,s+pt)$-core partitions when $p$ grows arbitrarily large. Of particular interest throughout is the comparison to the behavior of simultaneous $(s,t)$-cores. 
\end{abstract}

\section{Introduction}
A \emph{partition} $\lambda$ of a positive integer $n$ is a nonincreasing sequence of positive integers which sum to $n$.  We say that $n$ is the \emph{size} or \emph{weight} of the partition $\lambda$.  

Among the various types of integer partitions that have been studied, the $t$-core partitions are subjects of great interest and much research in combinatorics, representation theory, and related areas.  For a recent survey, see \cite{CKNS}.  The following is the \emph{Young} or \emph{Ferrers diagram} of the partition $\lambda = (5,3,2,2,1,1)$, with its \emph{hooklengths} marked: for each box in the diagram, the hooklength is the number of boxes directly right of and directly below the box, plus 1 for the box itself.

$$\young({1\!0}7421,{7 }41::,{5 }2:::,{4 }1:::,{2\,\,}::::,{1\,\,}::::)$$

One may see that the partition above lacks hooks of length 3, 6, 8, 9, and any number larger than 10.  A partition which lacks hooks of length $t$ is a $t$-\emph{core} partition, or $t$-\emph{core}  for short.  Lacking hooks of length $t_1, t_2, \dots$, a partition is a \emph{simultaneous} $(t_1,t_2,\dots)$-\emph{core}.



To situate the current paper in context, here are some of the important known results concerning cores and simultaneous cores. In order to state our generating function results below, we require the following notation, often referred to as $q$-Pochhammer notation.  For $|q|<1$, let 
$$
(A;q)_\infty = (1-A)(1-Aq)(1-Aq^2)(1-Aq^3)\dots .
$$
Note also that, for the rest of the paper, given integers $s$ and $t$, we let $d=\gcd(s,t)$.

\begin{theorem}[Olsson, \cite{Olsson}] The generating function for the number $c_t(n)$ of $t$-core partitions of size $n$ is 
$$\sum_{n=0}^\infty c_t(n) q^n= \frac{(q^t;q^t)_\infty^t}{(q;q)_\infty}.$$
\end{theorem}

\begin{theorem}[Anderson, \cite{Anderson}] If $d=1$, the number of $(s,t)$-cores is $$\frac{1}{s+t}\binom{s+t}{t}.$$
\end{theorem}

\begin{theorem}[Aukerman, Kane, and Sze, \cite{AKS}] 
\label{thm:AKS} 
If $d>1$, then the generating function for the number $c_{(s,t)}(n)$ of $(s,t)$-cores is 
$$C_{s,t}(q) = \sum_{n=0}^\infty c_{(s,t)}(n) q^n = \frac{(q^d;q^d)_\infty^d}{(q;q)_\infty} C_{s/d,t/d}(q^d)^d.$$\end{theorem}

Kane \cite{KaneThesis} conjectured, and Olsson and Stanton proved \cite{OlssonStanton}, that when $d = 1$, the largest $(s,t)$-core is of size $(s^2-1)(t^2-1)/24$.  While this gives its degree, we note that otherwise, for general coprime $s$ and $t$, little is known about the form of the polynomial $C_{s,t}(q)$.

Finally, if we increase the number of restrictions even less is known.  The $(t_1,t_2,t_3)$-cores are mostly studied for the case of $(s,s+t,s+2t)$-cores, and likewise for longer arithmetic progressions, the $(s,s+t,s+2t,\dots, s+pt)$-cores. For these an enumeration is known in the case of $s$ coprime to $t$, but there is no known general generating function.

\begin{theorem}[Cho, Huh, and Sohn, \cite{CHS}] 
\label{CHSthm}
If $d=1$, then the number of $(s,s+t,s+2t,\dots,s+pt)$-core partitions is $$\frac{1}{s+t}\binom{s+t}{t} + \sum_{k=1}^{\lfloor s/2 \rfloor} \sum_{\ell=0}^r \frac{1}{k+t}\binom{k+t}{k-\ell}\binom{k-1}{\ell}\binom{s+t-\ell(p-2)-1}{2k+t-1},$$
where $r=\min\left(k-1,\lfloor\frac{s-2k}{p-2}\rfloor\right)$.
\end{theorem}

In this paper we consider the $p \rightarrow \infty$ boundary case of $(s,s+t,s+2t,\dots, s+pt)$-cores.  If $s<t$, this means we are interested in partitions with hooks that avoid the residue class $s \pmod{t}$, and hence we name these $s \pmod{t}$-\emph{cores}, even when $s>t$.  (When the latter case arises we will be careful to note the elision.)  In this nomenclature one may consider $t$-cores to be $0 \pmod{t}$-cores.

When $s$ and $t$ are coprime, the set of $s \pmod{t}$-cores is finite, and the enumeration of the set is known by taking the limiting value for $p$ in the formula of Cho, Huh and Sohn.  Therefore a refined problem is to seek a description of generating functions of $s \pmod{t}$-cores.

Denote the number of $s \pmod{t}$-cores of $n$ by $c_{s(t)}(n)$, and their generating function by $C_{s(t)}(q) = \sum_{n=0}^\infty c_{s(t)}(n) q^n$.  We have the following.

\begin{lemma}\label{basiclemma} If $d = 1$, then $C_{s(t)}(q)$ is a polynomial. If $d>1$, then $C_{s(t)}(q)$ is an infinite series.\end{lemma}

\begin{theorem}\label{stgenfun} For any $s$ and $t$, $$C_{s(t)}(q) = \frac{(q^d;q^d)_\infty^d}{(q;q)_\infty} C_{s/d(t/d)}(q^d)^d.$$
\end{theorem}

We note, in passing, the similar structure of the generating functions given in Theorem \ref{thm:AKS} and Theorem \ref{stgenfun}.

The degree of $C_{s(t)}(q)$ can be given in the very restricted case $t=1$.  It could be derived as the large $p$-limit of a formula of Xiong \cite{Xiong} on the largest size of $(s,s+1,\dots,s+p)$-cores, but an elementary isoperimetric argument suffices to establish that the largest size of any partition avoiding hooks of length $s$ or greater is that of a square or near-square: either $\frac{s^2}{4}$ or $\frac{s^2-1}{4}$, as $s$ is even or odd respectively.

By Theorem \ref{stgenfun}, information about the coprime case yields generating functions for the non-coprime case.  Since the enumerating function in a non-coprime case is positive for an infinite number of input values, a type of property of greater interest for these is congruences for the number of $s \pmod{t}$-cores of size $An+B$.  We prove several such congruences below.

In Section \ref{FayersSec} of the paper, we consider the following conjecture of Fayers \cite{Fayers} on a surprising pattern in the enumeration of $s \pmod{t}$-cores.

\begin{conjecture}[Fayers \cite{Fayers}]\label{fayersconj} Suppose $t \geq 1$.  Then there is a monic polynomial $f_t(s)$ of degree $t-1$ with non-negative integer coefficients such that for any $s \geq 1$ coprime to $t$, the number of $(s,s+t,s+2t,\dots)$-cores is $$\frac{2^{s-t}f_t(s)}{t!}.$$ The constant term of $f_t(s)$ is $(2^t-1)(t-1)!$. \end{conjecture}

\noindent \textbf{Remark:} Assuming the conjecture, a finite number of data points can be used to fit initial cases of what the polynomials $f_t(s)$ must be:

\begin{align*}
f_1(s) &= 1 \\
f_2(s) &= s+3 \\
f_3(s) &= s^2 + 9s + 14 \\
f_4(s) &= s^3 + 18s^2 + 83s + 90 \\
\vdots &
\end{align*}

\noindent Fayers \cite{Fayers} gave examples of these polynomials $f_t(s)$ up to $t=11$ in factored form.  Upon multiplication, we located the coefficients in the Online Encyclopedia of Integer Sequences \cite{OEIS} at sequence number A079638.
We are able to show Conjecture \ref{fayersconj} holds.  We are also able to identify the polynomials involved.  We find that, in the rising-factorial basis instead of the standard basis, they have a startlingly concise presentation.

\begin{theorem}\label{fayersthm} Conjecture \ref{fayersconj} holds.  The polynomials required are given by
\begin{equation}
\label{FayersPolys}
f_t(s) = \sum_{m=1}^t L(t,m) (s+1)^{\overline{m-1}},
\end{equation} where the $L(t,m)$ are the unsigned Lah numbers defined for $t \geq m \geq 1$ by $$L(t,m) = \frac{t!}{m!} \binom{t-1}{m-1}$$ with $x^{\overline{0}} = 1$ for $x \geq 0$ and, for $m \geq 1$, $x^{\overline{m}} := x(x+1)(x+2) \dots (x+m-1)$.   (These expressions $x^{\overline{m}}$ are often referred to as \emph{rising factorials}.) 
\end{theorem}

We note here four additional results which will be pivotal in the work below.  We begin with a crucial result that allows us to manipulate generating functions in a straightforward manner with an eye towards proving congruences.  

\begin{theorem}
\label{DreamCong}
For a prime $p$ and positive integers $k$ and $\ell$, 
$$
(q^{\ell} ; q^{\ell})_\infty^{p^k} \equiv (q^{\ell p} ; q^{\ell p})_\infty^{p^{k-1}} \pmod{p^k}.
$$
\end{theorem}
\begin{proof} 
See \cite[Lemma 3]{dSS}.  
\end{proof}

\begin{theorem}[Euler's Pentagonal Number Theorem] 
\label{EulerPNT}
We have 
$$(q;q)_\infty = \sum_{n \in \mathbb{Z}} (-1)^n q^{n(3n-1)/2}.$$
\end{theorem}
\begin{proof} 
See \cite[Equation (1.6.1)]{Hir}.  
\end{proof}

\begin{theorem}[Jacobi]
\label{JacobiCube}
We have 
$$(q;q)_\infty^3 = \sum_{n \geq 0} (-1)^n (2n+1) q^{n(n+1)/2}.$$
\end{theorem}
\begin{proof} 
See \cite[Equation (1.7.1)]{Hir}.  
\end{proof}

\begin{theorem}
\label{XiaYao_2diss}
We have 
$$\frac{(q^9;q^9)_\infty}{(q;q)_\infty} = \frac{(q^{12};q^{12})_\infty^3(q^{18};q^{18})_\infty}{(q^{2};q^{2})_\infty^2(q^{6};q^{6})_\infty(q^{36};q^{36})_\infty} + q\frac{(q^{4};q^{4})_\infty^2(q^{6};q^{6})_\infty(q^{36};q^{36})_\infty}{(q^{2};q^{2})_\infty^3(q^{12};q^{12})_\infty}.$$
\end{theorem}
\begin{proof}
This 2--dissection result appears in the work of Xia and Yao \cite[Lemma 3.5]{XiaYao}.
\end{proof}

\begin{corollary}
\label{XYCor}
We have 
$$\frac{(q^9;q^9)_\infty}{(q;q)_\infty} \equiv_3  (q^{2};q^{2})_\infty^4 + q\frac{(q^{36};q^{36})_\infty}{(q^{4};q^{4})_\infty}$$
where the notation $X \equiv_m Y$ is shorthand for $X \equiv Y \pmod{m}$. 
\end{corollary}
\begin{proof}
This result follows immediately from elementary manipulations of the identity in Theorem \ref{XiaYao_2diss} using Theorem \ref{DreamCong}.
\end{proof}

To close out this section, we state the second part of Fayers' conjecture, which remains open.
\begin{conjecture}[Fayers, \cite{Fayers}]\label{fayersconj2}
Consider the same environment as Conjecture \ref{fayersconj}, and suppose $t \geq 2$. Then $f_t(s)$ is divisible by $s+t+(-1)^t$.
\end{conjecture}

\section{Generating functions and congruences}
\renewcommand\thetheorem{\thesection.\arabic{theorem}}
We begin this section by proving Lemma \ref{basiclemma} and Theorem \ref{stgenfun}.

\begin{proof}[Proof of Lemma \ref{basiclemma}] In the $p \rightarrow \infty$ limit, the Cho-Huh-Sohn formula from Theorem \ref{CHSthm} enumerates the $s \pmod{t}$-cores. One may also note that the set of $(s,s+t)$-cores, which is finite for coprime $s$ and $s+t$, contains the $s \pmod{t}$-cores, and $\gcd(s,s+t)=\gcd(s,t)=1$.  For $d>1$, the infinite set of $d$-cores is a subset of the $s \pmod{t}$-cores.
\end{proof}

\begin{proof}[Proof of Theorem \ref{stgenfun}]
This is a trivial statement when $d = 1$. For $d > 1$, the Littlewood decomposition into the $d$-core and $d$-quotient, as discussed by Olsson \cite[Proposition 3.6]{Olsson93}, ensures that the $d$-quotient consists of partitions that are exactly $s/d (t/d)$-cores. The form of the generating function provided follows.
\end{proof}
Note that the same logic applied to simultaneous $(s,t)$-core partitions produces Corollary 3 in \cite{AKS} and Corollary 1.8 in \cite{GNS}, which is why the corresponding generating functions bear a strong resemblance to one another.

Thanks to Theorem \ref{stgenfun}, we can write down many generating functions once we know the coprime cases. 

For $s=2$, $t$ odd, we observe that a $2 \pmod{t}$-core is in particular a 2-core, which are known to be the ``staircase'' partitions $ \varnothing, (1), (2,1), (3,2,1), \dots $ of size $\binom{n+1}{2}$, $n \geq 0$.  The largest hooklength in the partition $(m,m-1,\dots,1)$ is $2m-1$, and all odd hooklengths below this appear.  Hence we have the following.

\begin{theorem} For any $m\geq 0$,
$$C_{2(2m+1)}(q) = \sum_{k=0}^{m+1} q^{\binom{k+1}{2}}.$$
\end{theorem}

For example, the generating function for partitions with no hook of length $2 \pmod{3}$ is $1+q+q^3$, since the partitions $\varnothing$, $(1)$, and $(2,1)$ have hooks of length 1 and 3, while $(3,2,1)$ has a hook of length 5.

Hence, thanks to Theorem \ref{stgenfun}, we have the following.  

\begin{theorem}\label{4m+2} For any $m\geq 0$,
$$C_{4(4m+2)}(q) =  \frac{(q^2;q^2)_\infty^2}{(q;q)_\infty} \left(\sum_{k=0}^{m+1} q^{k^2+k} \right)^2 = \left( \sum_{n=0}^\infty q^{\binom{n+1}{2}} \right) \left(\sum_{k=0}^{m+1} q^{2 \binom{k+1}{2}} \right)^2.$$
\end{theorem}

\begin{example} Let $m=0,1,2,3$. Then Theorem \ref{4m+2} produces the following generating functions.
\begin{align*}
C_{4(2)}(q) &= \left( \sum_{n=0}^\infty q^{\binom{n+1}{2}} \right)(1+q^2)^2, \\
 C_{4(6)}(q) &=\left( \sum_{n=0}^\infty q^{\binom{n+1}{2}} \right) \left( 1 + q^2 + q^6 \right)^2, \\
 C_{4(10)}(q) &=\left( \sum_{n=0}^\infty q^{\binom{n+1}{2}} \right) \left( 1 + q^2 + q^6 + q^{12} \right)^2, \\
  C_{4(14)}(q) &=\left( \sum_{n=0}^\infty q^{\binom{n+1}{2}} \right) \left( 1 + q^2 + q^6 + q^{12} + q^{20} \right)^2.
\end{align*}
\end{example}

The first line here is an exception of the type we mentioned earlier; we refer here to partitions which lack hooks in the set $\{4,6,8,10,12,\dots\}$, i.e., any even hook \emph{except} those of length 2.  This is also the set of simultaneous $(4,6)$-cores.  In general the set of $4\pmod{4m+2}$-cores is the set of simultaneous $(4,4m+6)$-cores: the avoided hookset $\{4, 4m+6, 4m+10, \dots \}$ happens to be the semigroup generated by 4 and $4m+6$, because $2(4m+6) = 8m+12$, which is already divisible by 4.

This is \emph{not} the case when $d>2$, as we shall see later in this section. 

Briefly, if one is interested in congruences, one may note that the functions $C_{4(4m+2)}$ entirely avoid numerous arithmetic progressions, sometimes modulo 2 and sometimes identically. We provide the following example.

\begin{theorem}For all $j\geq 0$, it holds that $c_{4(2)}(11j+9) = 0$ and $c_{4(2)}(11j+2) \equiv 0 \pmod{2}$.
\end{theorem}

\begin{proof} For $N \equiv 9 \pmod{11}$ to appear as the exponent of $q$ in the expansion of $\left( \sum_{n=0}^\infty q^{\binom{n+1}{2}} \right)(1+2q^2+q^4)$, it must be of the form $T_n + k$ where $T_n$ is a triangular number and $k \in \{0,2,4\}$.  It is easy to check that the residues of triangular numbers modulo 11 are $\{0,1,3,4,6,10\}$.  Thus, modulo 11, the exponents of $q$ in the expansion of $\left( \sum_{n=0}^\infty q^{\binom{n+1}{2}} \right)(1+2q^2+q^4)$ must be in the set  $\{0,1,2,3,4,5,6,7,8,10\}$.  Hence $N \equiv 9 \pmod{11}$ is never represented.  Moreover, $N \equiv 2 \pmod{11}$ can appear as an exponent in this expansion only by adding 2 to a residue 0.  This occurs exactly by multiplication involving the term $2q^2$, and $2q^2 \equiv 0 \pmod{2}$.  This proves the second part of the theorem.
\end{proof}

In general, if $d=2$, then $C_{s(t)}(q) = \left( \sum_{n=0}^\infty q^{\binom{n+1}{2}} \right) f(q)$ for some polynomial $f(q)$, and for large enough moduli will appear as short isolated nonzero segments, eventually possessing the behavior above.

We can even make a statement which holds for infinitely many $t$.

\begin{theorem}\label{4(t)} For $t \equiv 2 \pmod{4}$, $p \equiv 3 \pmod{4}$ prime, $k \not\equiv 0 \pmod{p}$, and all $n\geq 0$, it holds that $$c_{4(t)}(p^2 n + k p - 2^{-1} - 2^{-3}) \equiv 0 \pmod{2},$$ where $2^{-1}$ is the inverse of $2$ modulo $p^2$.  
\end{theorem}
\begin{proof} From Theorem \ref{stgenfun}, we have \begin{align*} C_{4(4m+2)}(q) &= \left( \sum_{n=0}^\infty q^{\binom{n+1}{2}} \right)\left( \sum_{k=0}^{m+1} q^{2 \binom{k+1}{2}} \right)^2 \\
 &\equiv_2 \left( \sum_{n=0}^\infty q^{\binom{n+1}{2}} \right)\left( \sum_{k=0}^{m+1} q^{4 \binom{k+1}{2}} \right).\end{align*}
 
Hence any integer which has an odd number of $4 \pmod{4m+2}$ cores must be representable as the sum of a triangular number and 4 times a triangular number.  Choose a prime $p \equiv 3 \pmod{4}$ and let $N$ be so representable.  Then by completing the square we have that \begin{align*}N &\equiv 4 \binom{k+1}{2} + \binom{j+1}{2} \pmod{p^2} \\
 &\equiv 2(k+2^{-1})^2 + 2^{-1} (j+2^{-1})^2 - 2^{-1} - 2^{-3} \pmod{p^2} .
\end{align*}

Let $ r \equiv -2^{-1} - 2^{-3} \pmod{p^2}$ and suppose that $N \equiv r \pmod{p}$.  Then we must have for some $x, y \in \mathbb{Z}$ that $$2x^2 \equiv -2^{-1} y^2 \pmod{p}.$$

This holds if $x \equiv y \equiv 0 \pmod{p}$, in which case $N \equiv r \pmod{p^2}$.  If not, then neither $x$ nor $y$ can be divisible by $p$, and we may write $$\left( \frac{x}{y} \right)^2 \equiv -2^{-2} \pmod{p}.$$  But $-1$ is not a quadratic residue for $p \equiv 3 \pmod{4}$, and so this congruence has no nonzero solutions.  Hence the claimed residues cannot be represented by $T_k + 4T_j$ for any triangular numbers $T_k$ and $T_j$, and the theorem holds.
\end{proof}

\begin{example} The following initial examples of Theorem \ref{4(t)} are calculated for $p = 3$, in which case $r = 5$, and $p=7$, in which case $r = 30$: for all $n \geq 0$, 
\begin{align*}
c_{4(t)}(9n+2) &\equiv 0 \pmod{2} \, , \\ 
c_{4(t)}(9n+8) &\equiv 0 \pmod{2} \, , \text{\ \ and}\\ 
c_{4(t)}(49n+k) &\equiv 0 \pmod{2}
\end{align*}
for $k\in \{ 2,9,16,23,37,44\}$.
\end{example}
\noindent \textbf{Remark:} In fact, as these congruences hold by the same argument in the large-$t$ limit, these are refinements of known congruences for the 4-core partitions (see Hirschhorn and Sellers \cite{HS_EJC} for $c_4(9n+k)$ and Corollary 6 of Chen \cite{Chen} for $c_4(49n+k)$) under truncation of one factor in its generating function, for we have for the 4-core partitions 
\begin{align*}
\sum_{n=0}^\infty c_4(n) q^n &= \frac{(q^4;q^4)_\infty^4}{(q;q)_\infty} \\
&= \frac{(q^4;q^4)_\infty}{(q;q)_\infty} (q^4;q^4)_\infty^3 \\
&\equiv_2 \frac{(q^2;q^2)_\infty^2}{(q;q)_\infty} (q^4;q^4)_\infty^3 \text{\ \ \ using Theorem \ref{DreamCong} }\\
&\equiv_2 \left( \sum_{n=0}^\infty q^{\binom{n+1}{2}} \right)\left( \sum_{k=0}^\infty q^{4 \binom{k+1}{2}} \right) \text{\ \ \ using Theorem \ref{JacobiCube}. }
\end{align*}

Turning to the case $d=3$, we have the following.

\begin{theorem}\label{3mod1} For all $m\geq 0$, \begin{align*}C_{3(3m+1)}(q) &= -q^{3(m+1)^2+2(m+1)} + \sum_{j=0}^{m+1}q^{j^2+j} \sum_{\ell=-j}^{m+1}q^{\ell j + \ell^2 + \ell} \\ &= -q^{(3m+5)(m+1)} + \sum_{j=0}^{m+1} \sum_{\ell=-j}^{m+1} q^{2 \binom{j+\ell+1}{2} - \ell j}.\end{align*}
\end{theorem}

\begin{theorem}\label{3mod2} For all $m\geq 0$, \begin{align*}C_{3(3m+2)}(q) &= -q^{(3m+4)(m+2)} + \sum_{j=0}^{m+1}q^{j^2+j} \sum_{\ell=-j}^{m+2}q^{\ell j + \ell^2} .\end{align*}
\end{theorem}

\begin{proof} In order to state our proofs we require a brief overview of the \emph{abacus} of a partition. Mark the outer boundary of a partition with white \emph{spacers} on the horizontal unit segments and black \emph{beads} on the vertical segments.  Allow for an indefinite extension of black beads prior to the diagram and white spacers afterward.  Label the positions of beads and spacers with $\mathbb{Z}$, increasing to the north and east in the profile (i.e., from smaller toward larger parts), starting with 0 at the first spacer that appears in this orientation.  As an example, here we show the diagram described above for the partition $(5,3,3)$.

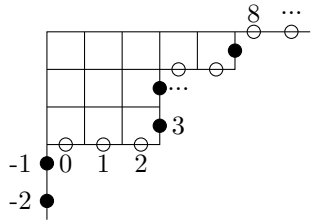
\begin{center}
\begin{tikzpicture}[scale=0.5]
\draw (0,-5) -- (0,0) -- (7,0);
\draw (0,-3) -- (3,-3) -- (3,0);
\draw (0,-2) -- (3,-2);
\draw (0,-1) -- (5,-1) -- (5,0);
\draw (1,0) -- (1,-3);
\draw (2,0) -- (2,-3);
\draw (4,0) -- (4,-1);
\draw (0.5,-3) circle (5pt)
(1.5,-3) circle (5pt)
(2.5,-3) circle (5pt)
(3.5,-1) circle (5pt)
(4.5,-1) circle (5pt)
(5.5,0) circle (5pt)
(6.5,0) circle (5pt);
\filldraw [black] (0,-4.5) circle (5pt)
(0,-3.5) circle (5pt)
(3,-2.5) circle (5pt)
(3,-1.5) circle (5pt)
(5,-0.5) circle (5pt);
\node at (-0.7,-4.5) {-2};
\node at (-0.7,-3.5) {-1};
\node at (0.5,-3.5) {0};
\node at (1.5,-3.5) {1};
\node at (2.5,-3.5) {2};
\node at (3.5,-2.5) {3};
\node at (3.5,-1.5) {...};
\node at (5.5,0.5) {8};
\node at (6.5,0.5) {...};
\end{tikzpicture}
\end{center}

Now straighten out the profile.  Since the partition starts with its first spacer and ends with its last bead, we have all information about the partition in its \emph{bead sequence}.

\begin{center}
\dots \begin{tikzpicture}[scale=0.5]
\draw (0,0) -- (12,0);
\filldraw [black] (0.5,0) circle (5pt)
(1.5,0) circle (5pt)
(5.5,0) circle (5pt)
(6.5,0) circle (5pt)
(9.5,0) circle (5pt);
\draw (2.5,0) circle (5pt)
(3.5,0) circle (5pt)
(4.5,0) circle (5pt)
(7.5,0) circle (5pt)
(8.5,0) circle (5pt)
(10.5,0) circle (5pt)
(11.5,0) circle (5pt);
\node at (0.5,0.5) {...};
\node at (1.5,0.5) {-1};
\node at (2.5,0.5) {0};
\node at (3.5,0.5) {1};
\node at (4.5,0.5) {2};
\node at (5.5,0.5) {3};
\node at (6.5,0.5) {4};
\node at (7.5,0.5) {5};
\node at (8.5,0.5) {6};
\node at (9.5,0.5) {7};
\node at (10.5,0.5) {8};
\node at (11.5,0.5) {...};
\end{tikzpicture} \dots
\end{center}

Finally, when we are interested in properties modulo $d$, it is useful to fold the abacus back on itself, taking positions $d$ at a time to create the $d$-\emph{runners}.  Place spacers and beads in positions $i \pmod{d}$ on the runner $i$, with labels from $jd+0$ to $jd+(d-1)$ aligned descending vertically, $j$ increasing to the right.  We illustrate here with the 3-abacus of the example partition.

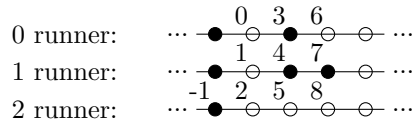
\begin{center}
\begin{tikzpicture}[scale=0.5]
\draw (-1,0) -- (4,0);
\draw (-1,-1) -- (4,-1);
\draw (-1,-2) -- (4,-2);
\draw (0.5,0) circle (5pt)
(0.5,-1) circle (5pt)
(0.5,-2) circle (5pt)
(1.5,-2) circle (5pt)
(2.5,0) circle (5pt)
(2.5,-2) circle (5pt)
(3.5,0) circle (5pt)
(3.5,-1) circle (5pt)
(3.5,-2) circle (5pt);
\filldraw [black] (1.5,0) circle (5pt)
(1.5,-1) circle (5pt)
(2.5,-1) circle (5pt)
(-0.5,0) circle (5pt)
(-0.5,-1) circle (5pt)
(-0.5,-2) circle (5pt);
\node at (-0.9,-1.5) {-1};
\node at (0.2,0.5) {0};
\node at (0.2,-0.5) {1};
\node at (0.2,-1.5) {2};
\node at (1.2,0.5) {3};
\node at (1.2,-0.5) {4};
\node at (1.2,-1.5) {5};
\node at (2.2,0.5) {6};
\node at (2.2,-0.5) {7};
\node at (2.2,-1.5) {8};
\node at (4.5, 0) {...};
\node at (4.5, -1) {...};
\node at (4.5, -2) {...};
\node at (-1.5, 0) {...};
\node at (-1.5, -1) { ...};
\node at (-1.5, -2) {...};
\node at (-4.5, 0) {0 runner: };
\node at (-4.5, -1) {1 runner: };
\node at (-4.5, -2) {2 runner: };
\end{tikzpicture}
\end{center}


To prove Theorems \ref{3mod1} and \ref{3mod2} we now consider the 3-runners.  The arguments are similar so we only go through the $3 \pmod{3m+1}$ case in detail.

The abacus below gives all possible places in which we may have beads in such a partition.  The ellipses in the middle indicate elided segments consisting of $m-2$ positions on each runner.  This means that the beads on runner 1 (the second row) extend $m+1$ positions past the beads on runner 0, and of beads on runner 2 extend $m$ positions beyond the beads on runner 1.  

\begin{center}
\dots
\begin{tikzpicture}[scale=0.5]
\draw (0,0) -- (9,0);
\draw (0,-1) -- (9,-1);
\draw (0,-2) -- (9,-2);
\node at (6,-0.7) [black] {...};
\node at (3,-0.7) [black] {...};
\draw (2.5,0) circle (5pt)
(3.5,0) circle (5pt)
(4.5,0) circle (5pt)
(5.5,0) circle (5pt)
(6.5,0) circle (5pt)
(7.5,0) circle (5pt)
(8.5,0) circle (5pt)
(5.5,-1) circle (5pt)
(6.5,-1) circle (5pt)
(7.5,-1) circle (5pt)
(8.5,-1) circle (5pt)
(8.5,-2) circle (5pt)
(7.5,-2) circle (5pt);
\filldraw [black] (0.5,0) circle (5pt)
(1.5,0) circle (5pt)
(0.5,-1) circle (5pt)
(1.5,-1) circle (5pt)
(2.5,-1) circle (5pt)
(3.5,-1) circle (5pt)
(4.5,-1) circle (5pt)
(0.5,-2) circle (5pt)
(1.5,-2) circle (5pt)
(2.5,-2) circle (5pt)
(3.5,-2) circle (5pt)
(4.5,-2) circle (5pt)
(5.5,-2) circle (5pt)
(6.5,-2) circle (5pt);
\end{tikzpicture}
\dots
\end{center}

We begin by noting that the partition must begin with a spacer in partition 0, and that a hook of length $y$ occurs exactly when a spacer at place $x$ is followed by a bead at place $x+y$; in particular, a hook of length 3 occurs when a spacer is followed by a bead on the same runner.   Since we wish to avoid hooks of length 3, inclusion of a spacer forces the remainder of the runner to consist of spacers, and so $3 \pmod{3m+1}$-cores can be parametrized by the positions of the first spacers on the 1 and 2 runners.

Having a spacer in position 0, and requiring the avoidance of hooks of length 3, $3m+4$, and $6m+5$, means that we must have spacers in these positions, which indeed are the positions of the spacers in the diagram.  Now suppose that we place a spacer no later than position $3m+4$ on runner 1, in position $3j+1$, where $0 \leq j \leq m+1$.  There must be a spacer in position $\ell = (3j+1) + (3m+4)$ on runner 2, meaning at most a ``height'' of $m+1$ places later, and possibly earlier.  All positions $(3j+1) +(6m+5)$ are already filled with spacers due to the spacer at position 0.

Furthermore, any position earlier than $3j+2$ is acceptable on runner 2, since all necessary spacers on runner 0 are already placed, and any position $(3j+2) + (6m+5)$ already has a spacer.  Hence any position $3\ell+2$ on runner 2 no later than $(3j+1) + (3m+4)$ is acceptable, and so $0 \leq \ell \leq m+1$ is a valid range of heights $\ell$ for the last spacer, with a single exception if $j = m+1$.  In that case, the $\ell = m+1$ choice is forbidden because the spacer at position 0 demands a spacer at position $6m+5$, so position $(3m+4)+(3m+4)$ cannot be the first spacer on the 2 runner.

We construct the generating function by reading the possible sizes of parts, which consist of the numbers of white spacers prior to any black bead.  The sum is simplified by taking the entire range of heights of $\ell$ relative to $j$, from $-j$ to $m+1$, and subtracting the improper case described above.  Parts appear with difference 1 or 2 in arithmetic progression, so the sums are the doubled triangular numbers claimed, with small shifts.  The result is the claimed generating function.

A similar argument establishes the $3 \pmod{3m+2}$ case.

\end{proof}

It is an immediate consequence of the abacus argument that $c_{3(3m+1)}(n) \leq c_{(3,3m+4)}(n)$, with the $3 \pmod{3m+1}$-cores being a strict subset of the simultaneous $(3,3m+4)$-cores.  In fact, the $3 \pmod{3m+1}$-cores miss exactly the one largest simultaneous $(3,3m+4)$-core, which has a largest hook of length $6m+5$.

\begin{example} Let $m=0,1,2.$ Then Theorem \ref{3mod1} produces the following:
\begin{align*}
C_{3(4)}(q) &= 1 + q + 2 q^2 + 2 q^4 + q^5 + 2 q^6 + 2 q^{10}, \\
C_{3(7)}(q) &= 1 + q + 2 q^2 + 2 q^4 + q^5 + 2 q^6 + q^8 + 2 q^9 + 2 q^{10} \\ &+ 2 q^{12} + q^{16} + 2 q^{17} + 2 q^{24}, \\
C_{3(10)}(q) &= C_{3(7)}(q) + 2 q^{14} + 2 q^{16} + 2 q^{20} + 2 q^{26} + q^{33} + 2 q^{34} + 2 q^{44} .
 \end{align*}
\end{example}
\begin{example}
So using Theorem \ref{stgenfun}, and the $C_{3(4)}$ and $C_{2(3)}$ calculated above, we can give, for instance, the following generating functions: 
$$C_{6(8)}(q) = \left( \sum_{n=0}^\infty q^{\binom{n+1}{2}} \right) \left( 1 + q^2 + 2 q^4 + 2 q^8 + q^{10} + 2 q^{12} + 2 q^{20} \right)^2,$$

 $$C_{6(9)}(q) =\left( \frac{(q^3;q^3)_\infty^3}{(q;q)_\infty} \right) \left( 1 + q^3 + q^{9} \right)^3,$$
 
 $$C_{9(12)}(q) =\left( \frac{(q^3;q^3)_\infty^3}{(q;q)_\infty} \right) \left( 1 + q^3 + 2 q^{6} + 2 q^{12} + q^{15} + 2 q^{18} + 2 q^{30} \right)^3.$$
 \end{example}
 The 3-cores have greater density than the 2-cores; while not every whole number possesses a 3-core, unlike the 2-cores the counting function $c_3(n)$ is nonzero on a set of positive density and grows indefinitely large. See \cite{HS_3cores} for more information related to arithmetic properties satisfied by $c_3(n)$. Related results on 3-cores can be found in \cite{BN} and \cite{HO}.  Heuristically, then, we would expect that congruences for the functions above are much less elementary to prove, and arise less often.  We can show the following.
 
 \begin{theorem} For all $n\geq 0$,  $$c_{6(9)}(16n+12) \equiv 0 \pmod{2},$$ and $$c_{6(9)}(400n+k) \equiv 0 \pmod{3}$$ for $k \in \{38, 88, 118, 168, 198, 248, 278, 328 \}.$
 \end{theorem}
 \begin{proof} For the congruence modulo 2, it is known (see Robbins \cite{Robbins}, Theorem 7) that the generating function for the 3-cores satisfies $$\frac{(q^3;q^3)_\infty^3}{(q;q)_\infty} \equiv_2 \sum_{n \in \mathbb{Z}} q^{n(3n-2)}.$$
One observes that $n(3n-2)$ may take residues 0, 1, 5, or 8 modulo 16, and the polynomial multiplier is $$\left( 1 + q^3 + q^{9} \right)^3 \equiv_2 1+q^3+q^6+q^{15}+q^{18}+q^{21}+q^{27}.$$  Among the possible residues modulo 16 for $$\{0, 1, 5, 8 \} + \{0, 2, 3, 5, 6, 11, 15 \},$$ the sum 9 modulo 16 does not appear.
 
For the congruences modulo 3, one has 
\begin{align*}
\frac{(q^3;q^3)_\infty^3}{(q;q)_\infty} 
&\equiv_3 
\frac{(q^9;q^9)_\infty}{(q;q)_\infty}  \text{\ \ \ using Theorem \ref{DreamCong}}\\
&\equiv_3 
(q^2;q^2)_\infty^4 + q \frac{(q^{36};q^{36})_\infty}{(q^4;q^4)_\infty} \text{\ \ \ using Corollary \ref{XYCor}} \\
&\equiv_3 
(q^2;q^2)_\infty^4 + q (q^8;q^8)_\infty^4 + q^5 \frac{(q^{144};q^{144})_\infty}{(q^{16};q^{16})_\infty},
\end{align*}
where the last congruence holds since $\frac{(q^{36};q^{36})_\infty}{(q^4;q^4)_\infty}$ is $\frac{(q^9;q^9)_\infty}{(q;q)_\infty}$ with the substitution $q \rightarrow q^4$.

Likewise, we observe that $$(1+q^3+q^9)^3 \equiv_3 1 + q^9 + q^{27}.$$

Since $\frac{(q^{144};q^{144})_\infty}{(q^{16};q^{16})_\infty}$ is a function of $q^{16}$, nonzero terms in $$q^5 \frac{(q^{144};q^{144})_\infty}{(q^{16};q^{16})_\infty}(1+q^9 + q^{27})$$ will appear with residue 5, 14, or 0 modulo 16; the residues listed in the theorem are all 6 or 8 modulo 16.  Likewise for $q (q^8;q^8)_\infty^4(1 + q^9 + q^{27})$, which is nonzero at residues 1, 10, or 28 modulo 8; none of the residues listed in the theorem qualify.  This leaves us to consider the terms that might arise from $(q^2;q^2)_\infty^4$.

Next, using Theorem \ref{DreamCong} and Theorem \ref{EulerPNT}, we see that 
\begin{align*}
(q^2;q^2)_\infty^4 
&\equiv_3 
(q^6;q^6)_\infty(q^2;q^2)_\infty \\
&\equiv_3
\left( \sum_{m \in \mathbb{Z}} (-1)^m q^{3m(3m-1)}  \right) \left( \sum_{n \in \mathbb{Z}} (-1)^n q^{n(3n-1)} \right).
\end{align*}
Observing the residues modulo 400 of $9m^2-3m$ and $3n^2-n$, and taking all possible sums, we observe that the residues listed in the theorem are missed, and the theorem is proved. \end{proof}
 

\noindent \textbf{Remark: } The above are simply the first examples of a large class of related theorems which are easy to find by analyzing residues for a given modulus.  One may analogously show that $c_{6(9)}(113n+89) \equiv 0 \pmod{2}$, or that $c_{9(12)}(56n+k) \equiv 0 \pmod{2}$ for $k \in \{13, 28, 47\}$.  A natural route of inquiry would be to seek a general theorem yielding an infinitude of such congruences; the octagonal numbers are a straightforward quadratic form, but the polynomial factor introduces a complication.
 
\section{Fayers' conjecture}\label{FayersSec}

We now consider Conjecture \ref{fayersconj}.

We begin with the formula of Cho, Huh, and Sohn given in Theorem \ref{CHSthm}.  In the large-$p$ limit, only the $\ell = 0$ term survives in the inner sum.  Their formula becomes 

$$g_t(s) := \sum_{n=0}^\infty c_{s(t)}(n) = \frac{1}{s+t}\binom{s+t}{t} + \sum_{k=1}^{\lfloor s/2 \rfloor} \frac{1}{k+t}\binom{k+t}{k}\binom{s+t-1}{2k+t-1}.$$ 
(We briefly note that in order to avoid the indefinite form $0 \cdot \infty$ we may consider $p$ to be arbitrarily large though not infinite.)  With a small change in the denominators, we see that we can absorb the initial term and write $$g_t(s) = \sum_{k=0}^{\lfloor s/2 \rfloor} \frac{1}{t} \binom{k+t-1}{t-1} \binom{s+t-1}{2k+t-1}.$$

Define $$g_t^{\ast}(s) = \frac{2^{s-t}}{t!} \sum_{m=1}^t L(t,m) (s+1)^{\overline{m-1}} .$$

Our proof makes use of automated theorem proving tools in the Maple software package, including the commands \emph{sumtools} and \emph{sumrecursion}.  (We note for those verifying these computations with Maple that $s^{\overline{m}}$ is implemented by ``pochhammer(s,m)''.)  The strategy is to show that, with the $t$ and $s$ considered as indeterminates, the $g_t(s)$ and $g_t^{\ast}(s)$ satisfy the same recurrence and initial conditions.

The binomial expression for $g_t(s)$ renders it straightforwardly accessible to \emph{sumtools} and \emph{sumrecursion}, which yield that $g_t(s)$ satisfies the second-order recurrence $$2t g_t(s) = (s+3(t-1)) g_{t-1}(s) - (t-2) g_{t-2} (s)$$ with initial conditions $$g_1(s) = \sum_{k=0}^{\lfloor s/2 \rfloor} \binom{s}{2k}$$ and $$g_2(s) = \sum_{k=0}^{\lfloor s/2 \rfloor} \frac{k+1}{2} \binom{s+1}{2k+1} .$$

The polynomials $f_t(s)$ defined in \eqref{FayersPolys} can also be analyzed by \emph{sumtools} and \emph{sumrecursion}, yielding that they satisfy the second-order recurrence $$f_t(s) = (s+3(t-1))f_{t-1}(s) - 2(t-1)(t-2) f_{t-2}(s)$$ with initial conditions $$f_1(s) = L(1,1) (s+1)^{\overline{0}} = 1 \cdot 1 = 1$$ and $$f_2(s) = L(2,1) (s+1)^{\overline{0}} + L(2,2) (s+1)^{\overline{1}} = 2 \cdot 1 + 1 \cdot (s+1) = s + 3.$$

We now note that $$f_t(s) = \frac{t!}{2^{s-t}} g_t^{\ast}(s).$$  Substituting this into the recurrence for $f_t(s)$ we have $$g_t^{\ast}(s) = (s+3(t-1)) \frac{(t-1)!}{2^{s-(t-1)}} \cdot \frac{2^{s-t}}{t!} g_{t-1}^{\ast} (s) - 2(t-1)(t-2) \frac{(t-2)!}{2^{s-(t-2)}} \cdot \frac{2^{s-t}}{t!}g_{t-2}^{\ast} (s) .$$

Cross-multiplying and cancelling factors from the factorials, we find that this is equivalent to 

$$2t g_t^{\ast}(s) = (s+3(t-1)) g_{t-1}^{\ast}(s) - (t-2) g_{t-2}^{\ast} (s),$$ as desired.  It remains to show that initial conditions match, which we require for $t=1$ and $t=2$.

We have that $g_1^{\ast}(s) = \frac{2^{s-1}}{1!} (1) = 2^{s-1}$.  We calculate $$g_1(s) = \sum_{k=0}^{\lfloor s/2 \rfloor} \binom{s}{2k} = \sum_{k=0}^{\lfloor s/2 \rfloor} \binom{s-1}{2k-1} + \binom{s-1}{2k} = \sum_{j=0}^{s-1} \binom{s-1}{j} = 2^{s-1}.$$

For $t=2$, we have that $$g_2^{\ast}(s) = 2^{s-3} (s+3) \quad , \quad g_2(s) = \sum_{k=0}^{\lfloor s/2 \rfloor} \frac{k+1}{2} \binom{s+1}{2k+1}.$$

Recall the Binomial Theorem which states that, for all $n\geq 0$, 
$$(1+x)^n = \sum_{k=0}^n \binom{n}{k} x^k.$$  Hence $$\frac{(1+x)^{s+1} - (1-x)^{s+1}}{2} = \sum_{k \geq 0} \binom{s+1}{2k+1} x^{2k+1}.$$

We differentiate this expression with respect to $x$ to obtain $$\frac{1}{2} \left( (s+1)(1+x)^s + (s+1) (1-x)^s \right) = \sum_{k \geq 0} (2k+1) \binom{s+1}{2k+1} x^{2k}.$$

Then sum the previous two lines to obtain

\begin{multline*} \frac{(1+x)^{s+1} - (1-x)^{s+1}}{2} + \frac{(s+1)(1+x)^s + (s+1) (1-x)^s}{2} \\ = \sum_{k \geq 0} \left[ x \binom{s+1}{2k+1} + (2k+1) \binom{s+1}{2k+1} \right] x^{2k}.
\end{multline*} 

Set $x=1$ in this expression to obtain $$\frac{2^{s+1} + (s+1) 2^s}{2} = \sum_{k \geq 0} (2k+2) \binom{s+1}{2k+1}.$$

Divide through by 4 to obtain $$g_2(s) = \sum_{k \geq 0} \frac{k+1}{2} \binom{s+1}{2k+1} = 2^{s-3} (s+3),$$ as desired.  This confirms that $$g_t(s) = \frac{2^{s-t}f_t(s)}{t!},$$ for $f_t(s)$ a polynomial with non-negative integer coefficients.
We now prove Fayers' claims for the properties of $f_t(s)$.  First, observe that the highest-degree term of $f$ appears uniquely in the final term of the sum, $$L(t,t) (s+1)^{\overline{t-1}} = \frac{t!}{t!}\binom{t-1}{t-1} (s+1) (s+2) \dots (s+t-1).$$  Hence the polynomial $f_t(s)$ is monic and of degree $t-1$.

The remaining claim we will show is that the constant term of $f_t(s)$ is $(2^t-1)(t-1)!$.  By expansion of each rising factorial, the constant term of $f_t(s)$ is

\begin{align*} 
CT \left( f_t(s) \right) 
&= \sum_{m=1}^t L(t,m) (m-1)! \\
&= \sum_{m=1}^t \frac{t!}{m!} \binom{t-1}{m-1} (m-1)! \\ 
&= \sum_{m=1}^t \frac{t!}{m!} \frac{(t-1)!}{(t-m)! (m-1)!} (m-1)!  \\
&= (t-1)! \sum_{m=1}^t \frac{t!}{m! (t-m)!} \\
&= (t-1)! \sum_{m=1}^t \binom{t}{m} \\
&= (t-1)! (2^t - 1).  \\
\end{align*}
\hfill $\Box$

\section{Remarks and Problems}

We close by sharing several remarks including some potential directions for future research.
 
One observes that when $\gcd(s,t)=1$, the sequence of finite sets of $s \pmod{jt}$-cores approaches the set of $s$-cores as $j \rightarrow \infty$ under the condition $\gcd(s,j)=1$.  To be precise, any set of $s \pmod{t}$-cores is contained within the sets of $s \pmod{t j}$-cores, and the sequence of $s \pmod{t j^k}$-cores is totally ordered by containment.  The limit of the union of initial elements of any infinite subsequence of this collection is the set of $s$-cores.  For example, the set of $3 \pmod{4}$-cores, which are the $(3,7,11,15,19,23\dots)$-cores, is contained within the set of $3 \pmod{8}$-cores, which are the $(3,11,19,27,\dots)$-cores, and within the set of $3 \pmod{20}$-cores, though the second is not contained within the third.   Thus we have a hierarchy of subsets which eventually exhausts the set of $s$-cores.  The hierarchy associated to any infinite sequence of $j$ coprime to $s$ will do so, with the sequences $\{j^k\}$ perhaps being distinguished among these.  Are there any structurally interesting properties of these nested sequences?
 
By examining the general properties of $\left( C_{3(3m+1)}(q^3) \right)^3$ one might find that some congruence holds for all $C_{3k((3m+1)k)}$, or more general cases, as we did for $C_{4(4m+2)}$.  What is the relation between congruences that arise in these generating functions for fixed $s$, and those for $s$-cores?  A congruence for $C_{s(t)}$ when $\gcd(s,t)=d$ can be interpreted as a congruence for a linear combination of $c_d(n-d k_i)$ for some set of values $k_i$; if all of these are divisible by $m$ then the sum will be, but what congruences exist that are not so inherited, if any?

Several questions regarding Fayers' conjecture may be enlightening to further pursue.  Conjecture \ref{fayersconj2} remains open. Moreover, a combinatorial proof relating the Cho-Huh-Sohn formula to Theorem \ref{fayersthm} would also be desirable.

\end{document}